\documentclass[3p,sort & compress]{elsarticle}
\usepackage{amssymb}
\usepackage{latexsym}%LATEX 的数学符号宏包
\usepackage{amsfonts}

\usepackage{amsthm}

\usepackage{amsbsy}
\usepackage{amsmath}
\usepackage{graphics,graphicx}%图片宏包
\usepackage{subfigure}%控制绘图子图宏包
\usepackage{multirow,multicol}%表中的跨行、列合并宏包
\usepackage{bm}%处理数学公式中的黑斜体
\usepackage{bbm}%粗体,更粗
\usepackage{color}% 支持彩色
\usepackage{float}%浮动表格
\usepackage{calc}
\usepackage{caption}
\usepackage{dsfont}
\usepackage{ifthen}
\usepackage{mathrsfs}
\usepackage[colorlinks,linkcolor=blue,citecolor=blue]{hyperref}
\usepackage{epstopdf}
\usepackage{epsfig}
\usepackage{algorithm}
\usepackage{algorithmic}
\usepackage{pifont}
\usepackage{natbib}
\usepackage{overpic}
\usepackage{psfrag}
\usepackage{rotating}
\usepackage{stmaryrd}
\usepackage{verbatim}
\usepackage{setspace}
\newtheorem{theorem}{Theorem}
\newtheorem{lemma}[theorem]{Lemma}
\newdefinition{remark}{Remark}
\newtheorem{corollary}[theorem]{Corollary}
\newdefinition{method}{Method}
\newdefinition{example}{Example}

\numberwithin{equation}{section}

\numberwithin{theorem}{section}
\journal{Journal of \LaTeX\ Templates}
\begin{document}
\begin{frontmatter}
	 \title{Single-step triangular splitting iteration method for a class of two-by-two block linear system}
	 \author[214ICNIG]{Jie Wu}\ead{xiaojiestudent@163.com}
	 \author[shjtmath]{Xi-An Li\corref{lxa}}\ead{lixian9131@163.com,lixa0415@sjtu.edu.cn}
	 \cortext[lxa]{Corresponding author.}
	 \address[214ICNIG]{214 Institute of China North Industries Group, Anhui Province, Bengbu 233000, China}
	 \address[shjtmath]{School of Mathematical Sciences, Shanghai Jiao Tong University, Shanghai 200240, China}
	 
    \begin{abstract}
    For solving a class of block two-by-two real linear system, a new single-step iteration method based on triangular splitting scheme is proposed in this paper. Then the convergence properties of this method  are carefully investigated. In addition, we determine its optimal iteration parameters and give the corresponding optimal convergence factor. It is worth mentioning that the SSTS iteration method is robust and superior to SBTS and PSBTS iteration methods under suitable conditions.  Finally, some numerical experiments are carried out to validate the theoretical results and evaluate this new method.
    \begin{keyword}
    Two-by-two; Positive semidefinite; SSTS; Parameters; Preconditioner
    \MSC[2010]65F10 \sep 65F50
    \end{keyword}
    \end{abstract}
\end{frontmatter}

\section{Introduction}\label{sec:1}
We consider the numerical solution for a class of block two-by-two linear system of the form
\begin{equation}\label{0101}
\mathcal{A}z=\left[\begin{array}{cc}W & -T \\T & W \\\end{array}\right]\left[\begin{array}{c}x \\y \end{array}\right]=\left[\begin{array}{c}p \\q \end{array}\right],
\end{equation}
where $W, T\in\mathbb{R}^{n\times n}$ are both symmetric positive semidefinite matrices and satisfy $\text{null}(W)\bigcap \text{null}(T) = \{\bm{0}\}$ with $\text{null}(\cdot)$ denoting the null space of a given matrix, $x,y\in\mathbb{R}^{n}$ are unknown vectors and $p,q\in\mathbb{R}^{n}$ are given vectors. The solution of linear system $\eqref{0101}$ is unique \cite{bai2013preconditioned} and it is a real equivalent formulation of the following complex symmetric system of linear equations
\begin{equation}\label{0102}
Au=b~\mathrm{with}~A = W+iT\in \mathbb{C}^{n\times n}~\mathrm{and}~ u=x+iy,b=p+iq,
\end{equation}
where $i=\sqrt{-1}$. This system \eqref{0101} or \eqref{0102} frequently stem from scientific and engineering fields \cite{Arridge1999Optical,Poirier2000Efficient,Bertaccini2004Efficient,Benzi2010Block}.

Lots of effective iterative methods have been proposed in the literatures for solving the linear
systems \eqref{0101}. For example, Bai et al. presented a HSS method \cite{bai2003hermitian} based on the Hermitian/skew-Hermitian (HS) splitting of the  matrix $A$, and developed a MHSS  method \cite{bai2010modified} in order to accelerate the convergence rate of HSS method. This MHSS method is algorithmically described in the following.

\begin{method}[The MHSS iteration method]
	\label{m1}
	Given an initial guess $x^{(0)}\in \mathbb{C}^{n}$, for $k=0, 1,2,\cdots$, until $\{x^{(k)}\}$ converges, compute
	\begin{equation*}
	\begin{cases}
	(\alpha I+W)x^{(k+1/2)}=(\alpha I-iT)x^{(k)}+b,\\
	(\alpha I+T)x^{(k+1)}=(\alpha I+iW)x^{(k+1/2)}-ib.
	\end{cases}
	\end{equation*}
	where $\alpha$ is a positive constant and $I$ is the identify matrix.
\end{method}
They have also proved that for any positive parameter $\alpha$, the HSS and  MHSS methods will converge unconditionally to the unique solution of the linear systems \eqref{0101} when $T$ is symmetric positive semi-definite. After that, some variants of the above two methods were generalized and discussed by many researchers, please see \cite{Bai2011On,Dehghan2013A,Wu2017A} and the references therein. In 2016, an efficient scale-splitting (SCSP) iteration method \cite{Hezari2016A} was constructed by multiplying a complex number $(\alpha-i)$ through both sides of \eqref{0102}, then some generalized versions of this method were developed\cite{Zheng2017Double,huang2019generalized}. These methods are considerable to deal with linear system \eqref{0102}.

To solve \eqref{0101}, a block preconditioned MHSS (PMHSS) iteration method \cite{bai2013preconditioned} and its alternating-direction version \cite{Wang2016Alternating} were presented following the idea of MHSS-like methods.Recent years, some other effective techniques to solve linear system \eqref{0101} are also proposed, such as C-to-R iteration methods \cite{Axelsson2000real,Benzi2010Block} and GSOR-like iteration  methods \cite{Salkuyeh2014Generalized,hezari2015preconditioned}. Very recently, Li et al. \cite{Li2018On} established a symmetric block triangular splitting (SBTS) method to linear
systems \eqref{0101} basing on upper and lower triangular splitting formulations for matrix $\mathcal{A}$. It can be simply expressed as follows.
\begin{method}[The SBTS iteration method]
	\label{m2}
	Given an initial vectors $x^{(0)}\in \mathbb{C}^{n}$ and $y^{(0)}\in \mathbb{C}^{n}$, and a real relaxation factor $\alpha>0$. For $k=0, 1,2,\cdots$, until $\left\{(x^{(k)^{T}}, y^{(k)^{T}})^{T} \right\}$ converges to the exact solution of \eqref{0101}, compute
	\begin{equation*}
	\begin{cases}
	Wx^{(k+1/2)}= Ty^{(k)}+ p,\\
	\alpha Wy^{(k+1/2)}=(\alpha-1)Wy^{(k)}-T x^{(k+1/2)}+ q,\\
	\alpha Wy^{(k+1)}=(\alpha-1)Wy^{(k+1/2)}-T x^{(k+1/2)}+ q ,\\
	Wx^{(k+1)}= Ty^{(k+1)}+ p,
	\end{cases}
	\end{equation*}
	where $\alpha$ is a positive constant and $I$ is the identify matrix.
\end{method}
 
Numerical testes indicate it is powerful than the HSS and MHSS methods to solve \eqref{0101}. Analogous to GSOR and PGSOR iteration methods, a preconditioned SBTS(PSBTS) iteration method was designed by Zhang et al.\cite{Zhang2018Preconditioned} and its performance is outstanding than the SBTS one under some restrictions. Viewing the SBTS and PSBTS methods, it is similar to the iterative scheme of SSOR method and has a symmetric structure involve with every iterative step. By making a study of the PGSOR and SSOR iteration methods, this SBTS method should have some room for improvement.

In this work, a single-step triangular splitting (SSTS) iteration method is developed for solving \eqref{0101} which stems from a class of complex symmetric linear system. Then, we investigate its convergence properties and provide the approaches of choosing the optimal iteration parameters. Moreover, this new method and PSBTS method have the same optimal convergence factor, but the former one is faster than the latter one.% for elapsed time of computer. 

The paper is structured as follows. In Section \ref{sec:02}, we derive the process of  establishing the SSTS iteration method for solving \eqref{0101}. In Section \ref{sec:03}, the convergence properties of the SSTS iteration method are discussed. In Section \ref{sec:05}, some numerical experiments are implemented to test this novel method. Finally, some brief conclusions are made in Section \ref{sec:06}.

\section{The SSTS iteration method}\label{sec:02}
In this section, we provide the process of establishing our SSTS iteration method to deal with linear system \eqref{0101}. According to the preconditioned technique \cite{hezari2015preconditioned, Zhang2018Preconditioned}, we firstly reconstruct \eqref{0101} by means of the following matrix
\begin{equation*}
\mathcal{P}_{\omega}=\left[\begin{array}{cc}\omega I&I\\ -I&\omega I\end{array}\right],
\end{equation*}
where $\omega$ is a positive real constant. Multiplying by matrix $\mathcal{P}_{\omega}$ on both sides of \eqref{0101} gives
\begin{equation}\label{0201}
\widetilde{\mathcal{A}}z := \left[\begin{array}{cc} \omega W+ T &-(\omega T-W)\\ \omega T- W & \omega W +  T \end{array} \right]\left[\begin{array}{c}x\\y\end{array}\right]=\left[\begin{array}{c}\omega p+q\\\omega q-p\end{array}\right]=:\widetilde{b}.
\end{equation}
For ease of discussion, we denote $\widetilde{W}_{\omega}=\omega W+ T, \widetilde{T}_{\omega} = \omega T- W$ and $\widetilde{p} = \omega p+q, \widetilde{q}=\omega q-p$ throughout this
paper. In light of the PSBTS iteration method, we split the coefficient matrix $\widetilde{\mathcal{A}}$ of \eqref{0201} into two matrices:
\begin{equation}\label{0202}
\widetilde{\mathcal{A}}=\left[\begin{array}{cc}  \widetilde{W}_{\omega} & O\\ \widetilde{T}_{\omega} & \alpha \widetilde{W}_{\omega} \end{array} \right] - \left[ \begin{array}{cc} O& \widetilde{T}_{\omega}\\ O & (\alpha-1)\widetilde{W}_{\omega} \end{array} \right]
:=\mathcal{M}_{\alpha,\omega}-\mathcal{N}_{\alpha,\omega},
\end{equation}
where $\alpha$ is a positive real parameter. Then, we consider the following SSTS iteration scheme according to the above splitting form.

\begin{equation}\label{0203}
\left[\begin{array}{c} x^{(k+1)}\\y^{(k+1)}\end{array} \right]=
\mathcal{H}_{\alpha,\omega}\left[\begin{array}{c} x^{(k)}\\y^{(k)}\end{array} \right]+\mathcal{G}_{\alpha,\omega}\left[ \begin{array}{c}  p \\ q \end{array} \right],
\end{equation}
where
\begin{equation*}\label{0204}
   \begin{aligned}
\mathcal{H}_{\alpha,\omega}&=\mathcal{M}^{-1}_{\alpha,\omega}\mathcal{N}_{\alpha,\omega}\\
& = \left[\begin{array}{cc}  \widetilde{W}_{\omega} & O\\ \widetilde{T}_{\omega} & \alpha \widetilde{W}_{\omega} \end{array} \right]^{-1} \left[ \begin{array}{cc} O& \widetilde{T}_{\omega}\\ O & (\alpha-1)\widetilde{W}_{\omega} \end{array} \right]\\
&=\left[\begin{array}{cc} O&\widetilde{W}_{\omega}^{-1}\widetilde{T}_{\omega}\\O&\frac{\alpha-1}{\alpha}I-\frac{1}{\alpha}\widetilde{W}_{\omega}^{-1}\widetilde{T}_{\omega}\widetilde{W}_{\omega}^{-1}\widetilde{T}_{\omega} \end{array}\right]\\
  \end{aligned}
\end{equation*}
is the iteration matrix and $\mathcal{G}_{\alpha,\omega}^{-1} = \mathcal{M}_{\alpha,\omega}$.

Analogous to classical stationary iteration scheme, we give the SSTS  iteration method to solve linear system \eqref{0101} by utilizing \eqref{0202}, it is described in the following.

\begin{method}[The SSTS iteration method]
	Given any two initial vectors $x^{(0)}, y^{(0)}\in \mathbb{R}^{n}$ and two real relaxation factors $\alpha, \omega>0$. Using the following procedures to update the iteration sequence $\left\{(x^{(k)^{T}}, y^{(k)^{T}})^{T} \right\}(k=0, 1, 2, \cdots)$, until it converges to the exact solution of \eqref{0101}.
	\begin{equation*}
	\begin{cases}
	\widetilde{W}_{\omega}x^{(k+1)}=\widetilde{T}_{\omega}y^{(k)}+\widetilde{p},\\
	\alpha \widetilde{W}_{\omega}y^{(k+1)}=(\alpha-1)\widetilde{W}_{\omega}y^{(k)}-\widetilde{T}_{\omega}x^{(k+1)}+\widetilde{q}.\\
	\end{cases}
	\end{equation*}
\end{method}

Since $W, T \in\mathbb{R}^{n\times n}$ are symmetric positive semi-definite and satisfy $\textup{null}(W)\cap \textup{null}(T) = \{\bm{0}\}$, the matrix $\widetilde{W}_{\omega}$ is symmetric positive definite, then each step of the SSTS iteration can be solved effectively using mostly real arithmetic either exactly by a Cholesky factorization or inexactly by  conjugate gradient and multigrid scheme.

In light of the matrix splitting form \eqref{0202}, the the block lower triangular matrix $\mathcal{M}_{\alpha,\omega}$ whose diagonal parts are
symmetric positive could be used as a preconditioner for the linear systems \eqref{0201}. At each step of applying the  preconditioner $\mathcal{M}_{\alpha,\omega}$ for a preconditioned
system, it is necessary to solve sequences of generalized residual equations of the form
\begin{equation} \label{0205}
\left[\begin{array}{cc}  \widetilde{W}_{\omega} & O\\ \widetilde{T}_{\omega} & \alpha \widetilde{W}_{\omega} \end{array} \right]\left[\begin{array}{c} e\\f  \end{array}\right] = \left[\begin{array}{c} r\\s  \end{array}\right],
\end{equation}
In the sequel, the matrix $\mathcal{M}_{\alpha,\omega}$ will be referred to as a SSTS-preconditioner and accelerate Krylov subspace methods such as GMRES. The system \eqref{0205} can be solved by the following steps:\\
(1) Solve $\widetilde{W}_{\omega}e = r$ for $e$;\\
(2) Solve $\widetilde{W}_{\omega}f = \frac{1}{\alpha}\left(s-\widetilde{T}_{\omega}e\right)$ for $f$.

\begin{remark}\label{r000}
	Observing the above procedures, we need to address two linear systems with the same coefficient matrix $\widetilde{W}_{\omega}$ at each iteration in our algorithm. In actual computations, $\widetilde{W}_{\omega}$ is generally sparse, one can use fast numerical algorithms to solve the linear systems  in steps (1) and (2).
\end{remark}

\section{Convergence discussion for the SSTS iteration method}\label{sec:03}
In this section, we turn to study the convergence properties of the SSTS iteration method. Firstly, some useful lemmas are introduced to support our theories.
%Firstly, some useful lemmas are introduced to support our theories to the convergence of the SSTS iteration method.

\begin{lemma}\label{L0301}
	Let  matrices $W, T \in\mathbb{R}^{n\times n}$ be symmetric positive semi-definite and satisfy ${null}(W)\cap null(T) = \{\bm{0}\}$. Then the matrices
	$\widetilde{W}_{\omega}$ and $ \widetilde{T}_{\omega}$ with a real constant $\omega>0$ are symmetric positive definite and symmetric, respectively.
\end{lemma}

\begin{lemma}\label{L0302}
	\textup{\cite{Salkuyeh2014Generalized}} Let  matrices $\widetilde{W}_{\omega}$, $\widetilde{T}_{\omega}\in\mathbb{R}^{n\times n}$ be symmetric positive definite and symmetric, respectively. Then the eigenvalues of the matrix $\widetilde{S}= \widetilde{W}_{\omega}^{-1}\widetilde{T}_{\omega}$ are all real.
\end{lemma}

\begin{lemma}\label{L0303}
	\textup{\cite{Hezari2016A}} Let  matrices $W, T \in\mathbb{R}^{n\times n}$ be symmetric positive semi-definite and satisfy ${null}(W)\cap null(T) = \{\bm{0}\}$. 
	%Also let $\omega$ be a positive constant. %$\widetilde{W}_{\omega}=\omega W+ T$ and $ \widetilde{T}_{\omega} = \omega T- W$. 
	If $\mu$ is an eigenvalue of $\widetilde{S}=\widetilde{W}_{\omega}^{-1}\widetilde{T}_{\omega}$ with $0<\omega\in\mathbb{R}$, then there is a generalized eigenvalue $\eta$ of matrix pair $(W,T)$ that satisfies  $\mu = \frac{\omega \eta-1}{\omega+\eta}$ and the spectral radius of $\widetilde{S}$ holds 
	\begin{equation}\label{0301}
	\rho(\widetilde{S}) = {\max}\left\{\frac{1-\omega \eta_{\min}}{\omega+\eta_{\min}},\frac{\omega \eta_{\max}-1}{\omega+\eta_{\max}}\right\},
	\end{equation}
	here and thereafter $\eta_{\min}$ and $\eta_{\max}$ are the extreme generalized eigenvalues of matrix pair $(W, T)$.
\end{lemma}

Based on the above lemmas, the following main conclusions about our new method are induced.

\begin{theorem}\label{T0304}
	Let  matrices $W, T \in\mathbb{R}^{n\times n}$ be symmetric positive semi-definite and satisfy ${null}(W)\cap null(T) = \{\bm{0}\}$. Also let $0<\alpha, \omega\in\mathbb{R}$ and $\widetilde{S}=\widetilde{W}_{\omega}^{-1}\widetilde{T}_{\omega}$. Assuming $\lambda$ is an eigenvalue of $\mathcal{H}_{\alpha,\omega}$, then $\lambda=0$ with multiplicity  $n$ and the remaining $n$ eigenvalues of $\mathcal{H}_{\alpha,\omega}$ satisfy the following equation:
	\begin{equation}\label{0302}
	\lambda-1+\frac{1+\mu_{i}^{2}}{\alpha}=0
	\end{equation}
	with $\mu_{i} (i = 1, 2, \cdots ,n)$ being the eigenvalues of $\widetilde{S}$. Furthermore, the spectral radius of $\mathcal{H}_{\alpha,\omega}$ satisfies
	\begin{equation}\label{0303}
	\rho(\mathcal{H}_{\alpha,\omega})=\max\left\{
	\left|1-\frac{1+\mu_{\min}^{2}}{\alpha}\right|,
	\left|1-\frac{1+\mu_{\max}^{2}}{\alpha}\right|\right\}.
	\end{equation}
	Here and thereinafter, $\mu_{\min}=\underset{\mu_{i}\in \textup{sp}(\widetilde{S})}{\min}\{|\mu_{i}|\}$ and $\mu_{\max}=\underset{\mu_{i}\in \textup{sp}(\widetilde{S})}{\max}\{|\mu_{i}|\}$ with  $\textup{sp}(\cdot)$ representing the spectrum for a given matrix.
\end{theorem}

\begin{proof}
	By Lemma \ref{L0302}, it is not difficult to verify that $\widetilde{S}=\widetilde{W}_{\omega}^{-1}\widetilde{T}_{\omega}$ has a spectral decomposition $\widetilde{S}=P\Lambda P^{-1}$, where $P\in \mathbb{R}^{n\times n}$ is an invertible matrix and $\Lambda=diag(\mu_{1},\mu_{2},\cdots,\mu_{n})$ is a diagonal matrix spanning by  the spectrum of $\widetilde{S}$. Since 
	\begin{equation*}
	   \hat{\mathcal{H}}_{\alpha,\omega} = \mathcal{V}^{-1}\mathcal{H}_{\alpha,\omega}\mathcal{V}=\left[
	   \begin{matrix}
	    O & \Lambda\\
	    O & \frac{\alpha-1}{\alpha}I-\frac{1}{\alpha}\Lambda^2
	   \end{matrix}
	   \right] ~~\text{with}~~\mathcal{V}=\left[
	   \begin{matrix}
	   P & O\\
	   O & P
	   \end{matrix}
	   \right],
	\end{equation*}
	is similar to $\mathcal{H}_{\alpha,\omega}$, then they have the same spectrum. Suppose $\lambda$ is the eigenvalue of $\mathcal{H}_{\alpha,\omega}$, we have
	\begin{equation}\label{0304}
	   \begin{aligned}
	       	det(\lambda I-\mathcal{H}_{\alpha,\omega})&=	det(\lambda I-\hat{\mathcal{H}}_{\alpha,\omega})\\
	       	&=det\left(\lambda I-\left[\begin{array}{cc} O&\Lambda\\O&\frac{\alpha-1}{\alpha}I-\frac{1}{\alpha}\Lambda^2 \end{array}\right]\right)\\
	       	&=\lambda^{n}det\left(\lambda I-\frac{\alpha-1}{\alpha}I+\frac{1}{\alpha}\Lambda^{2}\right)\\
	       	&=0.
	   \end{aligned}
	\end{equation}
	Here, $I$  represents the $n$-by-$n$ identity matrix. Obviously, $\lambda=0$ is an eigenvalue of $\mathcal{H}_{\alpha,\omega}$ with multiplicity $n$ and the remaining $n$ eigenvalues satisfy  \eqref{0302} with respect to $\mu_{i} (i = 1, 2, \cdots,n)$. Moreover, 
	\begin{equation}\label{0305}
	\ell(\mu_{i}^{2})=1-\frac{1+\mu_{i}^{2}}{\alpha}
	\end{equation}
	is a decreasing function with respect to $\mu^{2}_{i}$. According to the definition of spectral radius for a given matrix,  we clearly have \eqref{0303}.  
\end{proof}

In the following theorem, we derive the convergence domain of the SSTS iteration method.
\begin{theorem} \label{T0305}
	Let  matrices $W, T \in\mathbb{R}^{n\times n}$ be symmetric positive semi-definite and satisfy ${null}(W)\cap null(T) = \{\bm{0}\}$, $\omega>0$. Then the SSTS iteration method  is convergent if and only if
	\begin{equation}
	   \alpha>\frac{1+\mu^{2}_{\max}}{2}.
	\end{equation}
\end{theorem}

\begin{proof}
	In light of the Theorem \ref{T0304}, the eigenvalue of $\mathcal{H}_{\alpha,\omega}$ with $\lambda=0$ of $n$ multiple and the remaining $n$ eigenvalues are
	\begin{equation*}
	\lambda-1+\frac{1+\mu_{i}^{2}}{\alpha}=0.
	\end{equation*}
	If the SSTS iteration method is convergent, it should have $\rho(\mathcal{H}_{\alpha,\omega})<1$. That is
	\begin{equation}\label{0306}
	|\lambda| = \left |1-\frac{1+\mu_{i}^{2}}{\alpha}\right|<1,
	\end{equation}
	or equivalently,
	\begin{equation}\label{03061}
	\begin{cases}
	2\alpha>1+\mu_{i}^{2},\\
	1+\mu_{i}^{2}>0.
	\end{cases}
	\end{equation}
	%\begin{equation}\label{0310}
	%  \begin{cases}
	%  2\alpha>1+\mu_{i}^{2},\\
	%  1+\mu_{i}^{2}>0.
	%  \end{cases}
	%\end{equation}
	Using the Lemma \ref{L0301}, it holds if and only if $2\alpha>1+\mu^{2}_{\max}$.
\end{proof}
\begin{remark}\label{r02}
	From the above result and Theorem 2 in \cite{Zhang2018Preconditioned}, the SSTS iteration method has a larger convergent domain than PSBTS iteration method for parameter $\alpha$ and it is insensitive to parameter $\omega$. Thus, our method will be more robust than the PSBTS iteration method to solve linear system \eqref{0101}.
\end{remark}

Next, we provide the optimal selections of the parameters $\alpha$ and $\omega$, respectively, which minimize the spectral radius of the iterative matrix $\mathcal{H}_{\alpha,\omega}$ for SSTS iteration method. Meantime, we also give the corresponding optimal convergence factor for our method.
\begin{theorem} \label{T0306}
	Let the conditions of Theorems \textup{\ref{T0304} and \ref{T0305}} be satisfied. Then the optimal values of the relaxation parameters $\alpha$ and $\omega$ for the SSTS iteration method are given by
	\begin{equation}\label{0307}
	\alpha_{opt}=\frac{2+\mu^{2}_{\min}+\mu^{2}_{\max}}{2}~\textup{and}~\omega_{opt} = \frac{1-\eta_{\min}\eta_{\max}+\sqrt{(1+\eta^2_{\min})(1+\eta^2_{\max})}}{\eta_{\min}+\eta_{\max}},
	\end{equation}
	respectively, then the  corresponding optimal convergence factor is
	\begin{equation}\label{0308}
	\rho(\mathcal{H}_{\alpha_{opt},\omega_{opt}})=\frac{\mu^{2}_{\max}-\mu^{2}_{\min}}{2+\mu^{2}_{\min}+\mu^{2}_{\max}}.
	\end{equation}
\end{theorem}

\begin{proof}
	Based on \eqref{0302} and \eqref{0303}, we choose the optimal parameter $\alpha_{opt}$ by addressing the following problem
	\begin{equation*}
	\underset{\alpha}{\min}\underset{\mu_{i}\in \textup{sp}(\widetilde{S})}{\max}\left|1-\frac{1+\mu^{2}_{i}}{\alpha}\right|.
	\end{equation*}
	Let
	\begin{equation*}
	f{(\alpha)}=1-\frac{1+\mu^{2}_{\min}}{\alpha}~\textup{and}~g{(\alpha)}=1-\frac{1+\mu^{2}_{\max}}{\alpha},
	\end{equation*}
	then the optimal parameter $\alpha_{opt}$ is attained when $f{(\alpha)}=-g{(\alpha)}$. By simple calculations, we have the former result of \eqref{0307}. Substituting the first equality of \eqref{0307}  into \eqref{0303} can easily lead to \eqref{0308}. Since 
	\begin{equation*}
	h(\mu^2_{\min},\mu^2_{\max}) = \frac{\mu^{2}_{\max}-\mu^{2}_{\min}}{2+\mu^{2}_{\min}+\mu^{2}_{\max}}
	\end{equation*}
	is increasing and decreasing about $\mu^2_{\min}$ and $\mu^2_{\max}$, respectively, then we hope to choose a proper parameter $\omega$ which minimizes the $\mu^2_{\max}$. Based on the Theorem 1 of \cite{Hezari2016A}, the latter result of \eqref{0307} is easy to obtain.
\end{proof}

\begin{remark}\label{r03}
	For the SSTS and PSBTS iteration methods, they have the same optimal convergence factor according to the above theorem and Theorem 3 in \cite{Zhang2018Preconditioned}. The PSBTS iteration method need solving four sub-systems with coefficient matrix $\widetilde{W}_{\omega}$, but our method only need dealing with two sub-systems with respect to $\widetilde{W}_{\omega}$. Thus, the SSTS iteration method will be more practical and effective under certain situations.
\end{remark}

\begin{remark}\label{r04}
	The optimal parameter $\alpha_{opt}$ belongs to $[1+\frac{1}{2}\mu^{2}_{\max},1+\mu^{2}_{\max}]$ by \eqref{0307}, then it will be convenient for determining the range of parameter $\alpha_{opt}$ when the $\mu_{\max}$ is known.
\end{remark}

In section \ref{sec:02}, the splitting matrix $\mathcal{M}_{\alpha,\omega}$ can serve as a preconditioner and  its properties  are necessary to study. Here, we analyze and describe the spectral properties of the matrix $\mathcal{M}_{\alpha,\omega}^{-1}\widetilde{\mathcal{A}}$.

\begin{corollary}\label{C0301}
	Let $\widetilde{\mathcal{A}}\in \mathbb{R}^{2n\times 2n}$ be the nonsingular block two-by-two matrix defined in \eqref{0201}, with $\widetilde{W}\in \mathbb{R}^{n\times n}$ being symmetric positive definite and $\widetilde{T}\in \mathbb{R}^{n\times n}$ being symmetric. Also let $\alpha,\omega$ be two positive constants and $\mathcal{M}_{\alpha,\omega}$ be defined in \eqref{0202}. Then, the following results of the preconditioned matrix $\Gamma_{\alpha,\omega} = \mathcal{M}_{\alpha,\omega}^{-1}\widetilde{\mathcal{A}}$ hold.
	\begin{itemize}
		\item [(i)]$\Gamma_{\alpha,\omega}$ has an eigenvalue $1$ with multiplicity at least $n$;
		\item [(ii)] the remaining nonunit eigenvalues of $\Gamma_{\alpha,\omega}$ satisfy the equation
		\begin{equation*}
		\sigma = \frac{1+\xi}{\alpha},
		\end{equation*}
		where $\xi = \frac{v^{*}\widetilde{S}^{2}v}{v^{*}v}$ with $\widetilde{S}=\widetilde{W}_{\omega}^{-1}\widetilde{T}_{\omega}$.
	\end{itemize}
\end{corollary}

\begin{proof}
	According to \eqref{0201}, we have
	\begin{equation}\label{0313}
	\begin{aligned}
	\Gamma_{\alpha}= \mathcal{M}^{-1}\mathcal{A} &=\left[\begin{array}{cc}  \widetilde{W}_{\omega} & O\\ \widetilde{T}_{\omega} & \alpha \widetilde{W}_{\omega} \end{array} \right]^{-1}
	\left[\begin{array}{cc}\widetilde{W}_{\omega}&-\widetilde{T}_{\omega}\\\widetilde{T}_{\omega}&\widetilde{W}_{\omega}\end{array}\right]
	=\left(\begin{array}{cc} I&-\widetilde{S}\\ O&\frac{1}{\alpha}I+\frac{1}{\alpha}\widetilde{S}^{2}\end{array}\right)
	\end{aligned},
	\end{equation}
	where $\widetilde{S}=\widetilde{W}_{\omega}^{-1}\widetilde{T}_{\omega}$.
	It is clear that the matrix $\Gamma_{\alpha,\omega}$ has an eigenvalue $1$ with multiplicity $n$. Additionally, the remaining $n$ eigenvalues satisfy the following eigenvalue problem
	\begin{equation}\label{0314}
	\sigma v = \frac{1}{\alpha}(I+ \widetilde{S}^{2})v,
	\end{equation}
	premultiplying $\frac{v^{*}}{v^{*}v}$ on the bothsides of \eqref{0314}, we further have
	\begin{equation*}
	\sigma = \frac{1+ \xi}{\alpha}.
	\end{equation*}
	Lemma \ref{L0302} tells us that  $\alpha>0$ and $\xi\geq 0$, then the eigenvalue $\sigma$ of $\Gamma_{\alpha,\omega}$ is positive.
\end{proof}

%\begin{remark}\label{r03}
%Note that the optimal iterative parameter $\alpha_{opt}$ and the optimal convergence factor $\rho(\mathcal{H}_{\alpha_{opt}})$ are related to the eigenvalues of $S = W^{-1}T$, then it is necessary to solve a generalized eigenvalue problem $Tx = \mu Wx$ for obtaining the corresponding eigenvalues. However, solving a eigenvalue problem is tough and time-consuming when the problem size is large, then it will exert an unfavorable effect on the solution procedure of our method. Fortunately, the $\alpha_{opt}$ is only related to the $\mu_{max}$ and $\mu_{min}$ of $S$, we need only to solve them. Obviously, the eigenvalues of $\mu_{max}$ and $\mu_{min}$(if not equal 0) are the principal eigenvalue of $S$ and $S^{-1}$, respectively. One can obtain them by means of some numerical methods such as power method and inverse power method \cite{Golub1996Matrix}, respectively.
%\end{remark}

\section{Numerical experiments}\label{sec:05}
With two numerical examples being introduced, we test and verify the feasibility and efficiency of the SSTS iteration method for solving linear system \eqref{0101} in this section. Meantime, we compare their numerical results including iteration steps (denoted as IT) and elapsed CPU time in seconds (denoted as CPU) with those of the MHSS, SBTS, PGSOR, PSBTS and SSTS iteration methods. The numerical experiments are performed in MATLAB[version 9.0.0.341360 (R2016a)] with machine precision $10^{-16}$.

In our implementations, the initial guess is chosen to be zero vector and the iteration is terminated once the relative residual error satisfies
\begin{equation*}
\mathrm{RES}:=\frac{\left\|r^{(k)}\right\|_{2}}{\left\|r^{(0)}\right\|_{2}}<10^{-6}.
\end{equation*}
where $r^{(0)}=b$ and $r^{(k)}=b-\mathcal{A}z^{(k)}$ with $z^{(k)}=(x^{(k)^{T}}, y^{(k)^{T}})^{T}$ being the current approximant solution.
%%%%%%%%%%%%%%%%%%%%%%%%%%%%%%%%%%%%%%%%%%%%%%%%%%%%%%%%%%%%%%%%%%%%%%%%%%%%%%%%%%%%%%%%%%%%%%%%%%%%%%%%%%%%%%%%%%%%%%%%%%%
\begin{example}\label{E01}
\cite{bai2010modified,Salkuyeh2014Generalized} Consider the complex symmetric linear system of the form
\begin{equation}\label{0501}
\left[\left(K+\frac{3-\sqrt{3}}{\tau}I\right)+i\left(K+\frac{3+\sqrt{3}}{\tau}I\right)\right]u=b,
\end{equation}
where $\tau$ is the time step-size, and $K$ is the five-point centered difference approximation of negative Laplacian operator $L=-\Delta$ for homogeneous Dirichlet boundary conditions on uniform mesh in the unit square $[0,1]\times[0,1]$. Hence, $K=I_{m}\otimes V_{m} + V_{m} \otimes I_{m}$ with $V_{m}=h^{-2} \mathrm{tridiag}(-1,2,-1)\in \mathbb{R}^{m\times m}$, where $\otimes$  the Kronecker product symbol and $h=1/(m+1)$ is the discretization mesh-size. 

This complex symmetric linear systems arises in centered difference discretization of $R_{22}$-Pade approximations in the time integration of parabolic partial differential equations \cite{Axelsson2000real}. In this example, $K$ is an $n\times n$ block diagonal matrix with $n = m^{2}$. In our testes, we take $\tau=h$. Furthermore, we normalize coefficient matrix and righthand side of \eqref{0501} by multiplying both by $h^{2}$. We take
\begin{equation*}
W=K+\frac{3-\sqrt{3}}{\tau}I~~~\mathrm{and}~~~T=K+\frac{3+\sqrt{3}}{\tau}I
\end{equation*}
The right-hand vector $b$ is given with its $j$th entry
\begin{equation*}
b_{j}=\frac{(1-i)j}{\tau (1+j)^2},j=1, 2, ..., n.
\end{equation*}
\end{example}

%%%%%%%%%%%%%%%%%%%%%%%%%%%%%%%%%%%%%%%%%%%%%%%%%%%%%%%%%%%%%%%%%%%%%%%%%%%%%%%%%%%%%%%%%%%%%%%%%%%%%%%%%%%%%%%%%%%%%%%%%%
\begin{example} \label{E02}
\cite{bai2010modified,Benzi2010Block}
Consider the complex symmetric linear system of the form
\begin{equation}\label{0502}
[(-\theta^{2}M+K)+i(\theta C_{V}+C_{H}]u=b,
\end{equation}
where $M$ and $K$ are the inertia and the stiffness matrices, $C_{V}$ and $C_{H}$ are the viscous and hysteretic damping matrices, respectively. $\theta$ is the driving circular frequency and $K$ is defined the same as in Example \ref{E01}.

In this example, $K$ is an $n\times n$ block diagonal matrix with $n=m^{2}$. We choose $C_{H} = \varsigma K$ with $\varsigma$ being a damping coefficient, $M = I_{n}, C_{V} = 10I_{n}$. Additionally, we set $\theta=\pi$, $\varsigma=0.02$, and the righthand side vector $b$ is chosen such that the exact solution of the linear system \eqref{0502} is $b=(1 + i)A\bm{1}$. Similar to Example \ref{E01}, the linear system is normalized by multiplying both sides with $h^{2}$.
\end{example}

Firstly, the optimal iteration parameters of the MHSS, SBTS, PGSOR, PSBTS and SSTS iteration methods for Examples \ref{E01} and \ref{E02} are listed in Table \ref{TT01}. Expect for the MHSS iteration method, the parameters for tested methods are the theoretical optimal ones. The parameter of the SBTS, PGSOR and PSBTS iteration methods are chosen according to Theorem 3.5 in \cite{Li2018On}, Theorem 2.4 in \cite{hezari2015preconditioned} and Theorems 3 and 4 in \cite{Zhang2018Preconditioned}, respectively. In terms of the SSTS iteration method, we choose the theoretical optimal ones by the Theorem \ref{T0306} and also provide the experiential optimal ones at the same time.
%%%%%%%%%%%%%%%%%%%%%%%%%%%%%%%%%%%%%%%%%%%%%%%%%%%%%%%%%%%%%%%%%%%%%%%%%%%%%%%%%%%%%%%%%%%%%%%%%%%%%%%%%%%%%%%%%%%%%
\begin{table}[H]
  \centering
  \caption{ The parameters for the MHSS, SBTS, PGSOR, PSBTS and SSTS iteration methods}
  \label{TT01}
   \begin{tabular}{lccccccc}
     \hline
    \multirow{2}*{Example} &\multirow{2}*{Method}&                   &                             \multicolumn{5}{c}{Grid}                            \\ \cline{4-8}
                           &                    &             &$16\times16$&$32\times32$&$64\times64$&$128\times128$&$256\times256$ \\ \hline
\multirow{6}*{No.\ref{E01}}&MHSS&$\alpha_{opt}$               &1.06        &0.75        &0.54        &0.40          &0.30                   \\
						   &SBTS&$\alpha_{opt}$               &0.532       &0.525       &0.520       &0.518         &0.517                   \\
                           &PGSOR&$\alpha_{opt}/\omega_{opt}$ &0.990/0.657 &0.988/0.624 &0.986/0.602 &0.984/0.590   &0.983/0.583                  \\
                           &PSBTS&$\alpha_{opt}/\omega_{opt}$ &0.881/0.657 &0.864/0.624 &0.854/0.602 &0.849/0.590   &0.844/0.583        \\ 
      &       \multirow{2}*{SSTS}&$\alpha_{opt}/\omega_{opt}$ &1.019/0.657 &1.025/0.624 &1.030/0.602 &1.033/0.590   &1.035/0.583            \\ \vspace{0.075cm}
                           &    &$\alpha_{exp}/\omega_{exp}$  &1.04/0.601  &1.04/0.602  &1.045/0.605 &1.05/0.61     &1.05/0.61            \\ 
\multirow{6}*{No.\ref{E02}}&MHSS&$\alpha_{opt}$               &0.21        &0.08        &0.04        &0.02          &0.01                   \\
						   &SBTS&$\alpha_{opt}$               &11.986      &11.898      &11.875      &11.868        &11.863                   \\
                           &PGSOR&$\alpha_{opt}/\omega_{opt}$ &0.898/1.308 &0.896/1.324 &0.896/1.328 &0.895/1.330   &0.895/1.330                  \\
                           &PSBTS&$\alpha_{opt}/\omega_{opt}$ &0.689/1.308 &0.688/1.324 &0.687/1.328 &0.687/1.330   &0.687/1.330                \\
      &       \multirow{2}*{SSTS}&$\alpha_{opt}/\omega_{opt}$ &1.254/1.308 &1.259/1.324 &1.261/1.328 &1.262/1.330   &1.262/1.330     \\
                           &     &$\alpha_{exp}/\omega_{exp}$ &1.34/1.38   &1.38/1.32   &1.38/1.33   &1.40/1.33     &1.41/1.38      \\
     \hline
   \end{tabular}
\end{table}

Tables \ref{TT03} and \ref{TT04} show the numerical results including iteration step number and  CPU time to the above five methods with two examples. For our numerical implementation, the MHSS  ieration method is employed to cope with the complex symmetric linear system \eqref{0102}, while the SBTS, PGSOR, PSBTS and SSTS iteration  methods are employed to deal with the block two-by-two linear system \eqref{0101}.
%%%%%%%%%%%%%%%%%%%%%%%%%%%%%%%%%%%%%%%%%%%%%%%%%%%%%%%%%%%%%%%%%%%%%%%%%%%%%%%%%%%%%%%%%%%%%%%%%%%%%%%%%%%%%%%%%%%%%%%%%%%%%%%%
\begin{table}[H]
  \centering
  \caption{ Numerical results for Example \ref{E01}}
  \label{TT02}
   \begin{tabular}{lcccccccc}
     \hline  Method&                     &$16\times16$&$32\times32$&$64\times64$&$128\times128$&$256\times256$ \\  \hline
\multirow{2}*{MHSS}        &IT &40       &54          &73          & 98           & 133                         \\  \vspace{0.075cm}
                           &CPU&0.0165   &0.0727      &0.4019      &3.2004        & 23.0072                      \\  
%\multirow{2}*{SCSP} &IT &8           &9           &10          &10            &11                          \\
%                    &CPU&0.0069      &0.0141      &0.0526      &0.4205        &3.3611                       \\  \vspace{0.12cm}       
\multirow{2}*{SBTS}        &IT &24       &32          &39          &45            &48                        \\  \vspace{0.075cm}
                           &CPU&0.0168   &0.0806      &0.3929      &2.3942        &13.1121                            \\  
\multirow{2}*{PGSOR}       &IT &4        &4           &5           &5             &5                           \\   \vspace{0.075cm}
                           &CPU&0.0035   &0.0088      &0.0381      &0.1806        &0.9131                            \\  
\multirow{2}*{PSBTS}       &IT &4        & 4          &4           &4             &4                            \\  \vspace{0.075cm}
                           &CPU&0.0041   &0.0132      &0.0464      &0.2445        &1.2345                            \\ 
\multirow{2}*{SSTS$_{opt}$}&IT &4        & 5          &5           &5             &5                       \\     \vspace{0.075cm}
                           &CPU&0.0031   &0.0099      &0.0386      &0.1952        &0.9409                            \\   
\multirow{2}*{SSTS$_{exp}$}&IT &4        &4           &4           &4             &4                        \\
                           &CPU&0.0021   &0.0082      &0.0326     &0.1752        &0.8163                              \\                                  
     \hline
   \end{tabular}
\end{table}
%%%%%%%%%%%%%%%%%%%%%%%%%%%%%%%%%%%%%%%%%%%%%%%%%%%%%%%%%%%%%%%%%%%%%%%%%%%%%%%%%%%%%%%%%%%%%%%%%%%%%%%%%%%%%%%%%%%%%%%%%%
\begin{table}[H]
  \centering
  \caption{ Numerical results for Example \ref{E02}}
  \label{TT03}
   \begin{tabular}{lccccccc}
     \hline
  Method&                       &$16\times16$&$32\times32$&$64\times64$&$128\times128$&$256\times256$ \\  \hline
\multirow{2}*{MHSS}         &IT &34          &38          &50          &81            &139                          \\  \vspace{0.075cm}
                            &CPU&0.0153      &0.0261      &0.1182      &1.3595        &12.7184                       \\  
%\multirow{2}*{SCSP} &IT &8           &9           &10          &10            &11                          \\
%                    &CPU&0.0069      &0.0141      &0.0526      &0.4205        &3.3611                       \\  \vspace{0.12cm}       
\multirow{2}*{SBTS}         &IT & 78         & 77         & 77         & 77           & 77                          \\  \vspace{0.075cm}
                            &CPU&0.0185      &0.0686      &0.2725      &1.8316        & 9.9311                           \\  
\multirow{2}*{PGSOR}        &IT &8           &7           &8           &8             &8                           \\   \vspace{0.075cm}
                            &CPU&0.0016      &0.0055      &0.0231      &0.1277        &0.6957                            \\  
\multirow{2}*{PSBTS}        &IT &8           &9           &9           &9             &9                           \\  \vspace{0.075cm}
                            &CPU&0.0037      &0.0115      &0.0398      &0.2481        &1.4364                            \\ 
\multirow{2}*{SSTS$_{opt}$} &IT &9           &9           &10          &10            &10                           \\ \vspace{0.075cm}
                            &CPU&0.0026      &0.0069      &0.0286      &0.1665        &0.9509                            \\  
\multirow{2}*{SSTS$_{exp}$} &IT &8           &8           &7           &7             &6                           \\            
                            &CPU&0.0015      &0.0059      &0.0223      &0.1259        &0.6823                           \\  
   \hline
   \end{tabular}
\end{table}
%%%%%%%%%%%%%%%%%%%%%%%%%%%%%%%%%%%%%%%%%%%%%%%%%%%%%%%%%%%%%%%%%%%%%%%%%%%%%%%%%%%%%%%%%%%%%%%%%%%%%%%%%%%%%%%%%%%%%%
From Tables \ref{TT03} and \ref{TT04}, the following conclusions are clear. At first, the ITs of SSTS iteration method keep almost unchanged for theoretical and experiential optimal iteration parameters, respectively. Hence, the SSTS iteration method is stable with the problem size increasing. Secondly, the PGSOR, PSBTS and SSTS iteration methods outperform the MHSS and SBTS iteration methods for tested examples including both ITs and CPU time. It is neceassary to note that our method is slightly weaker than PGSOR iteration method for their theoretical optimal iteration parameters. It may be affected by the selection of parameters and the rounding error of computer, the IT is not same to SSTS and PSBTS iteration methods which is inconsonant to our analysis. However, the SSTS iteration method needs less CPU time than that of PSBTS. For experiential optimal iteration parameters, the SSTS iteration method  exceeds the other four methods in ITs and CPU time. Hence, our method with optimal parameters can be used to effectively solve \eqref{0101}.
\begin{table}[H]
	\centering
	\caption{ Numerical results for GMRES(10) with different preconditioners}
	\label{TT04}
	\begin{tabular}{lccccccc}
		\hline
		& &Method                  &$16\times16$&$32\times32$&$64\times64$&$128\times128$&$256\times256$\\ \hline
		&GMRES(10)     &IT         &5(4)   & 8(1)   &12(6)   &20(4)  &35(3)  \\\vspace{0.075cm}
		&              &CPU        &0.0108 & 0.0201 &0.0864  &0.3126 &1.7025        \\
		&MHSS-GMRES(10)&IT         &1(9)   &2(1)    &2(3)    &2(5)   &2(8)    \\\vspace{0.075cm}
		&              &CPU        &0.0115 &0.0348  &0.1512  &1.2079 &4.5305   \\
		&SBTS-GMRES(10)&IT         &1(8)   &1(10)   &2(2)    &2(6)   &2(7)    \\\vspace{0.075cm}
		&              &CPU        &0.0099 &0.0358  &0.1937  &1.2692 &6.6211         \\
		&PSBTS-GMRES(10)&IT        &1(3)   &1(4)    &1(4)    &1(4)   &1(4)     \\\vspace{0.075cm}
		&              &CPU        &0.0123 &0.0211  &0.0652  &0.4059 &1.9511    \\
		&PGSOR-GMRES(10)&IT        &1(4)   &1(4)    &1(4)    &1(4)   &1(4)       \\\vspace{0.075cm}
		&              &CPU        & 0.0065&0.0108  &0.0369  &0.2253 &0.9987        \\
		&SSTS$_{opt}$-GMRES(10)&IT &1(4)   &1(4)    &1(4)    &1(4)   &1(4)   \\\vspace{0.075cm}
		&              &CPU        &0.0051 &0.0109  &0.0382  &0.2202 &0.9967  \\
		&SSTS$_{exp}$-GMRES(10)&IT &1(4)   &1(4)    &1(4)    &1(5)   &1(5)   \\\vspace{0.075cm}
		&              &CPU        &0.0050 &0.0105  &0.0372  &0.2571 &1.1897  \\
		\hline
	\end{tabular}
\end{table}

\begin{table}[H]
	\centering
	\caption{ Numerical results for GMRES(10) with different preconditioners}
	\label{TT05}
	\begin{tabular}{lccccccc}
		\hline
		& &Method                  &$16\times16$&$32\times32$&$64\times64$&$128\times128$&$256\times256$\\ \hline
		&GMRES(10)     &IT         &5(4)   & 8(1)   &12(6)   &20(4)  &35(3)  \\\vspace{0.075cm}
		&              &CPU        &0.0108 & 0.0201 &0.0864  &0.3126 &1.7025        \\
		&MHSS-GMRES(10)&IT         &1(9)   &2(1)    &2(3)    &2(5)   &2(8)    \\\vspace{0.075cm}
		&              &CPU        &0.0115 &0.0348  &0.1512  &1.2079 &4.5305   \\
		&SBTS-GMRES(10)&IT         &1(8)   &1(10)   &2(2)    &2(6)   &2(7)    \\\vspace{0.075cm}
		&              &CPU        &0.0099 &0.0358  &0.1937  &1.2692 &6.6211         \\
		&PSBTS-GMRES(10)&IT        &1(3)   &1(4)    &1(4)    &1(4)   &1(4)     \\\vspace{0.075cm}
		&              &CPU        &0.0123 &0.0211  &0.0652  &0.4059 &1.9511    \\
		&PGSOR-GMRES(10)&IT        &1(4)   &1(4)    &1(4)    &1(4)   &1(4)       \\\vspace{0.075cm}
		&              &CPU        & 0.0065&0.0108  &0.0369  &0.2253 &0.9987        \\
		&SSTS$_{opt}$-GMRES(10)&IT &1(4)   &1(4)    &1(4)    &1(4)   &1(4)   \\\vspace{0.075cm}
		&              &CPU        &0.0051 &0.0109  &0.0382  &0.2202 &0.9967  \\
		&SSTS$_{exp}$-GMRES(10)&IT &1(4)   &1(4)    &1(4)    &1(5)   &1(5)   \\\vspace{0.075cm}
		&              &CPU        &0.0050 &0.0105  &0.0372  &0.2571 &1.1897  \\
		\hline
	\end{tabular}
\end{table}
Tables \ref{TT02} -- \ref{TT05} show that ITs of SSTS iteration method keep almost unchanged for theoretical and experiential optimal iteration parameters with mesh-grid increasing, respectively, it means our method is stable. In addition, they all outperform the MHSS, SBTS and PSBTS when they serve as solvers or preconditoners.  Noting that our method is slightly weaker than PGSOR iteration method for their theoretical optimal iteration parameters, but the experiential one is superior.  Hence, our method with optimal parameters can be used to effectively solve linear system \eqref{0101}.

\section{Conclusion}\label{sec:06}
We present a practical and effective single-step triangular splitting method for a class of block two-by-two linear system \eqref{0101} and investigate its convergence properties. Under suitable convergence conditions, the optimal iteration parameters and corresponding convergence factor of this method are also derived. For solving a class of block two-by-two system which arises from complex symmetric linear system, numerical experiments show that this novel method is more powerful comparing with the MHSS, SBTS and PSBTS iteration methods and is competed with the PGSOR iteration method. Moreover, when the SSTS is used as a preconditioner to Krylov subspace method such as GMRES($\ell$), it will improve obviously computing efficiency of the GMRES($\ell$) method.

%\section*{Acknowledgements}
%X.L and L.Z are partially supported by the National Natural Science Foundation of China (NSFC 11871339, 11861131004).

\bibliographystyle{model1-num-names}
\bibliography{References}

\begin{thebibliography}{19}
\expandafter\ifx\csname natexlab\endcsname\relax\def\natexlab#1{#1}\fi
\providecommand{\url}[1]{\texttt{#1}}
\providecommand{\href}[2]{#2}
\providecommand{\path}[1]{#1}
\providecommand{\DOIprefix}{doi:}
\providecommand{\ArXivprefix}{arXiv:}
\providecommand{\URLprefix}{URL: }
\providecommand{\Pubmedprefix}{pmid:}
\providecommand{\doi}[1]{\href{http://dx.doi.org/#1}{\path{#1}}}
\providecommand{\Pubmed}[1]{\href{pmid:#1}{\path{#1}}}
\providecommand{\bibinfo}[2]{#2}
\ifx\xfnm\relax \def\xfnm[#1]{\unskip,\space#1}\fi
%Type = Article
\bibitem[{Bai et~al.(2013)Bai, Benzi, Chen, and Wang}]{bai2013preconditioned}
\bibinfo{author}{Z.-Z. Bai}, \bibinfo{author}{M.~Benzi},
  \bibinfo{author}{F.~Chen}, \bibinfo{author}{Z.-Q. Wang},
\newblock \bibinfo{title}{{Preconditioned MHSS iteration methods for a class of
  block two-by-two linear systems with applications to distributed control
  problems}},
\newblock \bibinfo{journal}{IMA J. Numer. Anal.} \bibinfo{volume}{33}
  (\bibinfo{year}{2013}) \bibinfo{pages}{343--369}.
%Type = Article
\bibitem[{Arridge(1999)}]{Arridge1999Optical}
\bibinfo{author}{S.~R. Arridge},
\newblock \bibinfo{title}{{Optical tomography in medical imaging}},
\newblock \bibinfo{journal}{Inverse Probl.} \bibinfo{volume}{15}
  (\bibinfo{year}{1999}) \bibinfo{pages}{41--93}.
%Type = Article
\bibitem[{Poirier(2000)}]{Poirier2000Efficient}
\bibinfo{author}{B.~Poirier},
\newblock \bibinfo{title}{{Efficient preconditioning scheme for block
  partitioned matrices with structured sparsity}},
\newblock \bibinfo{journal}{Numer. Linear Algebra Appl.} \bibinfo{volume}{7}
  (\bibinfo{year}{2000}) \bibinfo{pages}{715--726}.
%Type = Article
\bibitem[{Bertaccini(2004)}]{Bertaccini2004Efficient}
\bibinfo{author}{D.~Bertaccini},
\newblock \bibinfo{title}{{Efficient preconditioning for sequences of
  parametric complex symmetric linear systems}},
\newblock \bibinfo{journal}{Electr. Trans. Numer. Anal. Etna}
  \bibinfo{volume}{18} (\bibinfo{year}{2004}) \bibinfo{pages}{49--64}.
%Type = Article
\bibitem[{Benzi and Bertaccini(2010)}]{Benzi2010Block}
\bibinfo{author}{M.~Benzi}, \bibinfo{author}{D.~Bertaccini},
\newblock \bibinfo{title}{{Block preconditioning of real-valued iterative
  algorithms for complex linear systems}},
\newblock \bibinfo{journal}{IMA J. Numer. Anal.} \bibinfo{volume}{28}
  (\bibinfo{year}{2010}) \bibinfo{pages}{598--618}.
%Type = Article
\bibitem[{Bai et~al.(2003)Bai, Golub, and Ng}]{bai2003hermitian}
\bibinfo{author}{Z.-Z. Bai}, \bibinfo{author}{G.~H. Golub},
  \bibinfo{author}{M.~K. Ng},
\newblock \bibinfo{title}{{Hermitian and skew-Hermitian splitting methods for
  non-Hermitian positive definite linear systems}},
\newblock \bibinfo{journal}{SIAM J. Matrix Anal. Appl.} \bibinfo{volume}{24}
  (\bibinfo{year}{2003}) \bibinfo{pages}{603--626}.
%Type = Article
\bibitem[{Bai et~al.(2010)Bai, Benzi, and Chen}]{bai2010modified}
\bibinfo{author}{Z.-Z. Bai}, \bibinfo{author}{M.~Benzi},
  \bibinfo{author}{F.~Chen},
\newblock \bibinfo{title}{{Modified HSS iteration methods for a class of
  complex symmetric linear systems}},
\newblock \bibinfo{journal}{Computing} \bibinfo{volume}{87}
  (\bibinfo{year}{2010}) \bibinfo{pages}{93--111}.
%Type = Article
\bibitem[{Bai et~al.(2011)Bai, Benzi, and Chen}]{Bai2011On}
\bibinfo{author}{Z.-Z. Bai}, \bibinfo{author}{M.~Benzi},
  \bibinfo{author}{F.~Chen},
\newblock \bibinfo{title}{{On preconditioned MHSS iteration methods for complex
  symmetric linear systems}},
\newblock \bibinfo{journal}{Numer. Algorithms} \bibinfo{volume}{56}
  (\bibinfo{year}{2011}) \bibinfo{pages}{297--317}.
%Type = Article
\bibitem[{Dehghan et~al.(2013)Dehghan, Dehghanimadiseh, and
  Hajarian}]{Dehghan2013A}
\bibinfo{author}{M.~Dehghan}, \bibinfo{author}{M.~Dehghanimadiseh},
  \bibinfo{author}{M.~Hajarian},
\newblock \bibinfo{title}{{A generalized preconditioned MHSS method for a class
  of complex symmetric linear systems}},
\newblock \bibinfo{journal}{Math. Model. Anal.} \bibinfo{volume}{18}
  (\bibinfo{year}{2013}) \bibinfo{pages}{561--576}.
%Type = Article
\bibitem[{Wu et~al.(2017)Wu, Li, and Yuan}]{Wu2017A}
\bibinfo{author}{Y.-J. Wu}, \bibinfo{author}{X.~Li}, \bibinfo{author}{J.-Y.
  Yuan},
\newblock \bibinfo{title}{{A non-alternating preconditioned HSS iteration
  method for non-Hermitian positive definite linear systems}},
\newblock \bibinfo{journal}{Comput. Appl. Math.} \bibinfo{volume}{36}
  (\bibinfo{year}{2017}) \bibinfo{pages}{367--381}.
%Type = Article
\bibitem[{Hezari et~al.(2016)Hezari, Salkuyeh, and Edalatpour}]{Hezari2016A}
\bibinfo{author}{D.~Hezari}, \bibinfo{author}{D.~K. Salkuyeh},
  \bibinfo{author}{V.~Edalatpour},
\newblock \bibinfo{title}{{A new iterative method for solving a class of
  complex symmetric system of linear equations}},
\newblock \bibinfo{journal}{Numer. Algorithms} \bibinfo{volume}{73}
  (\bibinfo{year}{2016}) \bibinfo{pages}{1--29}.
%Type = Article
\bibitem[{Zheng et~al.(2017)Zheng, Huang, and Peng}]{Zheng2017Double}
\bibinfo{author}{Z.~Zheng}, \bibinfo{author}{F.-L. Huang},
  \bibinfo{author}{Y.-C. Peng},
\newblock \bibinfo{title}{{Double-step scale splitting iteration method for a
  class of complex symmetric linear systems}},
\newblock \bibinfo{journal}{Appl. Math. Lett.} \bibinfo{volume}{73}
  (\bibinfo{year}{2017}).
%Type = Article
\bibitem[{Huang et~al.(2019)Huang, Wang, Xu, and Cui}]{huang2019generalized}
\bibinfo{author}{Z.-G. Huang}, \bibinfo{author}{L.-G. Wang},
  \bibinfo{author}{Z.~Xu}, \bibinfo{author}{J.-J. Cui},
\newblock \bibinfo{title}{{The generalized double steps scale-SOR iteration
  method for solving complex symmetric linear systems}},
\newblock \bibinfo{journal}{J. Comput. Appl. Math.} \bibinfo{volume}{346}
  (\bibinfo{year}{2019}) \bibinfo{pages}{284--306}.
%Type = Article
\bibitem[{Wang and Lu(2016)}]{Wang2016Alternating}
\bibinfo{author}{T.~Wang}, \bibinfo{author}{L.-Z. Lu},
\newblock \bibinfo{title}{{Alternating-directional PMHSS iteration method for a
  class of two-by-two block linear systems}},
\newblock \bibinfo{journal}{Appl. Math. Lett.} \bibinfo{volume}{58}
  (\bibinfo{year}{2016}) \bibinfo{pages}{159--164}.
%Type = Article
\bibitem[{Axelsson and Kucherov(2000)}]{Axelsson2000real}
\bibinfo{author}{O.~Axelsson}, \bibinfo{author}{A.~Kucherov},
\newblock \bibinfo{title}{{Real valued iterative methods for solving complex
  symmetric linear systems}},
\newblock \bibinfo{journal}{Numer. Linear Algebra Appl.} \bibinfo{volume}{7}
  (\bibinfo{year}{2000}) \bibinfo{pages}{197--218}.
%Type = Article
\bibitem[{Salkuyeh et~al.(2014)Salkuyeh, Hezari, and
  Edalatpour}]{Salkuyeh2014Generalized}
\bibinfo{author}{D.~K. Salkuyeh}, \bibinfo{author}{D.~Hezari},
  \bibinfo{author}{V.~Edalatpour},
\newblock \bibinfo{title}{{Generalized SOR iterative method for a class of
  complex symmetric linear system of equations}},
\newblock \bibinfo{journal}{Int. J. Comput. Math.} \bibinfo{volume}{92}
  (\bibinfo{year}{2014}).
%Type = Article
\bibitem[{Hezari et~al.(2015)Hezari, Edalatpour, and
  Salkuyeh}]{hezari2015preconditioned}
\bibinfo{author}{D.~Hezari}, \bibinfo{author}{V.~Edalatpour},
  \bibinfo{author}{D.~K. Salkuyeh},
\newblock \bibinfo{title}{{Preconditioned GSOR iterative method for a class of
  complex symmetric system of linear equations}},
\newblock \bibinfo{journal}{Numer. Linear Algebra Appl.} \bibinfo{volume}{22}
  (\bibinfo{year}{2015}) \bibinfo{pages}{761--776}.
%Type = Article
\bibitem[{Li et~al.(2018)Li, Zhang, and Wu}]{Li2018On}
\bibinfo{author}{X.-A. Li}, \bibinfo{author}{W.-H. Zhang},
  \bibinfo{author}{Y.-J. Wu},
\newblock \bibinfo{title}{{On symmetric block triangular splitting iteration
  method for a class of complex symmetric system of linear equations}},
\newblock \bibinfo{journal}{Appl. Math. Lett.} \bibinfo{volume}{79}
  (\bibinfo{year}{2018}) \bibinfo{pages}{131--137}.
%Type = Article
\bibitem[{Zhang et~al.(2018)Zhang, Wang, and Zhao}]{Zhang2018Preconditioned}
\bibinfo{author}{J.~Zhang}, \bibinfo{author}{Z.~Wang},
  \bibinfo{author}{J.~Zhao},
\newblock \bibinfo{title}{{Preconditioned symmetric block triangular splitting
  iteration method for a class of complex symmetric linear systems}},
\newblock \bibinfo{journal}{Appl. Math. Lett.}  (\bibinfo{year}{2018})
  \bibinfo{pages}{95--102}.

\end{thebibliography}

\end{document}